\documentclass{amsart}
\usepackage{amsfonts,amssymb,amsthm,amsmath}
\numberwithin{equation}{section}

\title[Relations in Quantum \(K\)-Theory of Flag Varieties]{
Minuscule Relations in Quantum \(K\)-theory \\
 of
 flag varieties
}
\author[K.~Brahma, Y.~Huang, T.~Ikeda, T.~Kouno, K.~Yamaguchi]{Koushik Brahma, Yicen Huang, Takeshi Ikeda, \\ Takafumi Kouno, Kohei Yamaguchi}
\dedicatory{
Waseda University, Faculty of Science and Engineering\\
3-4-1 Okubo, Shinjuku-ku
Tokyo 169-8555, Japan\\
Email addresses: koushikbrahma95@gmail.com, huangyicen@akane.waseda.jp \\
 gakuikeda@waseda.jp, t.kouno@aoni.waseda.jp, yamaguchi\_86@aoni.waseda.jp
}

\date{\today}

\usepackage{mathtools}
\mathtoolsset{showonlyrefs}
\usepackage[mathscr]{euscript}

\newcommand{\C}{\mathbb{C}}
\newcommand{\CO}{\mathscr{O}}
\newcommand{\I}{\mathbf{I}}

\newcommand{\af}{\mathrm{af}}

\newcommand{\Wafpos}{W_\af^{\ge 0}}

\newcommand{\la}{\lambda}
\newcommand{\Rep}{R(T)}

\renewcommand{\L}[1]{{L}(#1)}

\newcommand{\COQl}[1]{\mathscr{O}_{\mathbf{Q}_G}(#1)}
\newcommand{\COQ}[1]{\mathscr{O}_{\mathbf{Q}_G(#1)}}

\newcommand{\fh}{\mathfrak{h}}

\newcommand{\fg}{\mathfrak{g}}

\newcommand{\QG}{\mathbf{Q}_G}
\newcommand{\KTQG}{K_T(\QG)}
\newcommand{\vpi}[1]{{\varpi_{#1}}}

\renewcommand{\t}[1]{\tau_{#1}}
\newtheorem{thm}{Theorem}[section]
\newtheorem{lem}[thm]{Lemma}

\theoremstyle{definition}
\newtheorem{example}[thm]{Example}

\theoremstyle{remark}
\newtheorem{remark}[thm]{Remark}
\newcommand{\FZ}{Z}%Z{\mathfrak{Z}}

\begin{document}


\begin{abstract}
We study the quantum \(K\)-theory of the flag variety \(G/B\). For each
minuscule fundamental weight \(\varpi\), we construct an explicit relation
in the torus-equivariant quantum \(K\)-theory \(QK_T(G/B)\). The relation
can be regarded as a quantum deformation of the character of the
irreducible representation with highest weight \(\varpi\).
\end{abstract}
\maketitle

\section{Introduction}
Let $G$ be a simply connected simple complex algebraic group,
$B\subset G$ a Borel subgroup, and $T\subset B$ a maximal torus.
Set $\fg=\operatorname{Lie}(G)$, and let \(R(T)\) denote the
representation ring of \(T\). Let $QK_T(G/B)$ denote the
$T$-equivariant (small) quantum $K$-ring of the flag variety $G/B$, with Novikov variables $Q_i$
indexed by the simple coroots (see \cite{Givental} and \cite{Lee}).
Our goal is to construct, for each minuscule fundamental weight
$\varpi$, an explicit relation in $QK_T(G/B)$ associated with the
character of the irreducible $\fg$-module of highest weight $\varpi$.

Let $P$ denote the weight lattice of $\mathfrak g$. For $\lambda\in P$, let
\(\mathbb C_\lambda\) denote the one-dimensional \(B\)-module of weight
\(\lambda\), on which the unipotent radical of \(B\) acts trivially. 
We
denote by
\[
\CO_{G/B}(\lambda):=G\times^B \mathbb C_{-\lambda}
\]
the corresponding \(G\)-equivariant line bundle on \(G/B\).

Let \(I\) be the index set of the simple roots
\(\{\alpha_i\}_{i\in I}\), and
write $\alpha^\lor$ for the coroot corresponding to a root $\alpha.$ 
For \(i\in I\), let \(\varpi_i\) denote the \(i\)-th fundamental weight.
For \(\lambda\in P\), define
\begin{equation*}
    \L{\la}
    :=
    \prod_{i\in I}
    \CO_{G/B}(-\varpi_i)^{-\langle \lambda,\alpha_i^\lor\rangle}
    \in QK_T(G/B).
\end{equation*}
Here we use the same notation for a line bundle and its \(K\)-class,
and all products and inverses are taken with respect to the quantum product.
In particular, $L(-\varpi_i)=\CO_{G/B}(-\varpi_i).$
Although \(\L{\lambda}\) specializes to \(\CO_{G/B}(\lambda)\)
at \(Q_i=0\) for all \(i\in I\), it need not equal the line bundle
class \(\CO_{G/B}(\lambda)\) in \(QK_T(G/B)\).

For a dominant integral weight \(\lambda\), let \(V(\lambda)\)
denote the irreducible \(\fg\)-module of highest weight \(\lambda\),
and let \(\chi_\lambda\in R(T)^W\) denote its character,
where \(W\) is the Weyl group of \(G\).
We can regard \(\chi_\lambda\) as an element of \(QK_T(G/B)\)
via its \(R(T)\)-algebra structure.
An integral weight $\lambda$ is called \emph{minuscule} if
  $\langle \lambda,\alpha^\vee\rangle \in \{-1,0,1\}$
  for every positive root $\alpha$. 
If $\varpi$ is a minuscule fundamental weight, 
we have
$$
\chi_{\varpi}=\sum_{\mu\in W\cdot \varpi}
e^\mu.
$$
Our main result is the following.

\begin{thm}\label{thm:main}
Let \(\varpi\) be a minuscule fundamental weight. In \(QK_T(G/B)\), we have
\begin{equation}
    \label{eq:main}
    \sum_{\mu\in W\cdot\varpi}
    \prod_{\substack{i\in I\\
        \langle\mu,\alpha_i^\lor\rangle={-}1}}
        (1-Q_i)
\;\,    \L{-\mu}
    =
    \chi_\varpi .
\end{equation}
\end{thm}
At \(Q_i=0\) for all \(i\in I\), this identity reduces to
the relation associated with \(\chi_\varpi\) in the classical
Borel presentation of the $T$-equivariant $K$-theory ring \(K_T(G/B)\) of $G/B$; see McLeod \cite{Mc}.

In type \(\mathrm A_{n-1}\), all the fundamental weights
\(\varpi_i\), \(i=1,\ldots,n-1\), are minuscule. 
The corresponding relations coincide with the defining relations
in the Borel presentation of \(QK_T(SL_n(\mathbb C)/B)\)
established by Maeno, Naito, and Sagaki \cite{MNS1}. (See also Koroteev, Pushkar, Smirnov, and Zeitlin \cite{KPSZ}
for a related presentation in type \({\mathrm A}_{n-1}\)
obtained via quasimaps to quiver varieties.)
In type \(\mathrm C_n\), $\varpi_1$ is minuscule,  and 
the corresponding relation
coincides with one of the defining
relations in the Borel presentation of
\(QK_T(Sp_{2n}(\mathbb C)/B)\) given by Kouno and Naito \cite{KN}. Thus, our theorem gives relations in types \(\mathrm B_n\),
\(\mathrm D_n\), \(\mathrm E_6\), and \(\mathrm E_7\), and recovers
the known relations in types \(\mathrm A_{n-1}\) and \(\mathrm C_n\). We refer the reader to Section~\ref{sec:example}
for explicit examples of these relations in low rank.

\medskip

\noindent\textbf{Acknowledgements.}
We are especially grateful to Satoshi Naito and Daisuke Sagaki for bringing the results
of \cite{NS} to our attention.
We thank 
Hiroshi Iritani, Shinsuke Iwao, Hayato Koike, Yuan-Pin Lee,
Leonardo Mihalcea, Daniel Orr, Mark Shimozono,
Jun'ichi Shiraishi, and Kanehisa Takasaki for valuable discussions. This work was supported by JSPS KAKENHI Grant Numbers
24KF0258, 25KF0074, 
23K25772, 22K03239, 
26K16968. 
K.~Brahma was supported by a JSPS Postdoctoral Fellowship for Research in Japan, during which this research was carried out.

\section{Semi-infinite flag variety and Kato's isomorphism}\label{sec:semi-Kato}
In this section, we recall the semi-infinite flag variety \(\QG\) and its
Schubert classes. We follow 
the convention of \cite[\S~3]{MNS1} and
\cite[\S~2]{O}. 
We then explain Kato's isomorphism \cite{K}, which allows us to
reformulate the main theorem as an identity in \(\KTQG\).

\subsection{Basic notation}

We write $
  \fg=\operatorname{Lie}(G),
  \fh=\operatorname{Lie}(T).
$
Let $\Phi\subset\fh^*$ be the root system of $(G,T)$.
The choice of $B$ determines the set 
of positive roots.
We write
$
  \Delta:=\{\alpha_i\mid i\in I\}
$
for the corresponding set of simple roots.
For each $\alpha\in\Phi$, let $\alpha^\lor$ denote the corresponding
coroot.  We use the natural pairing
$  \langle\,\cdot\,,\,\cdot\,\rangle
  \colon P\times Q^\lor\longrightarrow\mathbb Z,
$
where $P$ is the weight lattice and
$  Q^\lor=\sum_{i\in I}\mathbb Z\alpha_i^\lor,
  \;
  Q^{\lor,+}=\sum_{i\in I}\mathbb Z_{\geq 0}\alpha_i^\lor.
$
The fundamental weights are denoted by $\varpi_i$, $i\in I$, and are
characterized by
$ \langle\varpi_i,\alpha_j^\lor\rangle=\delta_{ij}.
$

Let $W$ be the Weyl group of $G$.  For $\alpha\in\Phi$, let
$s_\alpha\in W$ denote the corresponding reflection, and put
$s_i=s_{\alpha_i}$.  We denote by $w_\circ$ the longest element of
$W$. 
We write
$W_{\mathrm{af}}=W\ltimes Q^\lor$
for the affine Weyl group.  For $\xi\in Q^\lor$, let $t_\xi$ denote
the corresponding translation element, and set
\[
  W_{\mathrm{af}}^{\geq 0}
  :=
  \{wt_\xi\mid w\in W,\ \xi\in Q^{\lor,+}\}.
\]
We identify the representation ring of $T$ with the group algebra
$ R(T)=\mathbb Z[P]
      =\bigoplus_{\lambda\in P}\mathbb Z e^\lambda.
$
For each $i\in I$, let $Q_i$ denote the Novikov variable corresponding
to $\alpha_i^\lor$.  For
$\xi=\sum_{i\in I}d_i\alpha_i^\lor\in Q^{\lor,+}$, we write
 $ Q^\xi:=\prod_{i\in I}Q_i^{d_i}.
$
\subsection{Semi-infinite flag variety}

The main reference for this subsection is \cite[\S~3.1]{MNS1},
which is based on \cite[\S~1.4--1.5]{K}.

Let $\QG$ denote the semi-infinite flag variety of $G$, and let
$\I\subset G(\C[\![t]\!])$ be the inverse image of $B$ under the
evaluation map at $t=0$.
The $\I$-orbits in $\QG$ are indexed by
$W_{\mathrm{af}}^{\geq0}$.
For $x\in\Wafpos$, we denote the closure of the corresponding
$\I$-orbit by $\QG(x)$ and call it a \emph{semi-infinite Schubert
variety}.  In particular, $\QG=\QG(e)$, where $e$ is the identity element of $W$.
Let $\COQ{x}$ denote the structure
sheaves of the semi-infinite Schubert varieties $\QG(x)$.

The $T$-equivariant $K$-group of $\QG$ is defined by
\[
  \KTQG
  :=
  \prod_{x\in\Wafpos}
  \Rep\,\COQ{x},
\]
where we simply denote
the class of the structure sheaf by the same symbol
$\COQ{x}$.
Thus, $\KTQG$ consists of all formal, possibly infinite,
$\Rep$-linear combinations of $\COQ{x}$, $x\in \Wafpos.$

For each $\la\in P$, we denote the associated $T$-equivariant line bundle on $\QG$
by $\COQl{\la}$ (see \cite[\S~3.1]{MNS1}).
\subsection{Kato's isomorphism}
There is an isomorphism between the torus-equivariant \(K\)-group of the
semi-infinite flag variety and the torus-equivariant quantum \(K\)-theory of
\(G/B\). We will use this isomorphism to translate identities in
\(K_T(\QG)\) into identities in \(QK_T(G/B)\).

Let
\[
  QK_T(G/B)
  :=R(T)[\![Q]\!]
  \otimes_{R(T)}K_T(G/B),
  \quad
  R(T)[\![Q]\!]
  :=
  R(T)[\![Q_i\mid i\in I]\!],
\]
where $Q_i=Q^{\alpha_i^\lor}$.
The ring $QK_T(G/B)$ is equipped with the quantum product.  For $w\in W$, let $\CO_{G/B}^w$ denote the class of the structure
sheaf of the opposite Schubert variety indexed by $w$.

\begin{thm}[\cite{K}, see also {\cite[\S~5]{MNS1}}]
\label{thm:Kato}
There is an isomorphism of abelian groups
\begin{equation}
\Psi: 
K_T(\QG)
\rightarrow
QK_T(G/B)
\end{equation}
such that 
\begin{itemize}
    \item[(i)] for $\la\in P,\;\xi\in Q^{\lor,+},\;
w\in W$, we have
\begin{equation}
\Psi(    e^{-\lambda}\COQ{wt_\xi})=e^\lambda Q^\xi 
\CO_{G/B}^{w}; \label{eq:Psi_COQ}
\end{equation} 
\item[(ii)] for $i \in I$ and ${Z} \in \KTQG$, we have 
\begin{equation} \label{eq:comm_fundamental}
    \Psi(\CO_{\QG}(w_{\circ} \vpi{i}) \otimes {Z}) = \CO_{G/B}(-\vpi{i}) \cdot \Psi({Z}). 
\end{equation}
\end{itemize}
\end{thm}

 \begin{remark}
     Since the conventions of Kato's semi-infinite flag manifold and our one are different, we need $w_{\circ}$ in the left-hand side of \eqref{eq:comm_fundamental}. 
 \end{remark}

\section{Proof of Theorem \ref{thm:main}}
\begin{lem}
\label{lem:Phi_sends_O_to_L}
For $\la\in P$ and any $\FZ \in \KTQG$, we have
\begin{equation}
\label{eq:Phi_sends_O_to_L}
\Psi(\COQl{-w_\circ\lambda} \otimes \FZ)
=\L{\la} \cdot \Psi(\FZ).\end{equation}
\end{lem}
\begin{proof}
For $i\in I$, let $A_i$ be the automorphism of $K_T(\QG)$ given by
tensor product with $\COQl{w_\circ\vpi{i}}$, and let $M_i$ be the
endomorphism of $QK_T(G/B)$ given by quantum multiplication by
$\CO_{G/B}(-\vpi{i})$.  Equation~\eqref{eq:comm_fundamental} says that
$\Psi\circ A_i=M_i\circ\Psi$.  Since $A_i$ is invertible, so is $M_i$,
and the same identity holds for every integral power of these
operators.

Write $\lambda=\sum_{i\in I}
\langle\lambda,\alpha_i^\lor\rangle\vpi{i}$.  Since
\[
\COQl{-w_\circ\lambda}
=
\bigotimes_{i\in I}
\COQl{w_\circ\vpi{i}}^{\otimes
-\langle\lambda,\alpha_i^\lor\rangle},
\]
successive application of \eqref{eq:comm_fundamental} and its inverse
gives
\[
\Psi(\COQl{-w_\circ\lambda}\otimes\FZ)
=
\prod_{i\in I}
\CO_{G/B}(-\vpi{i})^{-\langle\lambda,\alpha_i^\lor\rangle}
\cdot\Psi(\FZ)
=\L{\lambda}\cdot\Psi(\FZ).
\]
\end{proof}

For $\xi\in Q^{\lor,+}$, we define an
$\Rep$-linear endomorphism $\t{\xi}$ of $\KTQG$ by
\begin{equation}
\t{\xi}(\COQ{x})=\COQ{xt_{\xi}}.
\end{equation}

\begin{lem}[Naito--Sagaki {\cite[(3.11)]{NS}}] \label{lem:minuscule_correspondence}
If $\mu\in P$ is minuscule, then we have 
\begin{equation}
\Psi\left(
\prod_{\substack{i \in I \\ \langle\mu, \alpha_{i}^{\vee}\rangle = -1}} (1-\t{\alpha_i^\lor}) \CO_{\QG}(w_{\circ}\mu)\right)=
\CO_{G/B}(-\mu). 
\end{equation}
\end{lem}
\begin{remark}
The above result can also be deduced using interpolated quantum
Lakshmibai--Seshadri (IQLS for short) paths \cite{KN24} together with the Chevalley formula
for $\QG$ by 
Lenart, Naito, and Sagaki \cite[Theorem 33]{LNS-Selecta}.
If $\mu$ is minuscule, then all IQLS paths of shape $\mu$ are straight-line paths, see \cite[Remark~4.4]{KN24}. This directly implies \cite[Corollary~3.5]{NS}. 
\end{remark}

\begin{lem}\label{cor_Omu_Lmu}
    If $\mu$ is a minuscule fundamental weight then $$\CO_{G/B}(-\mu)=\prod_{\substack{i\in I\\
        \langle\mu,\alpha_i^\lor\rangle={-}1}}
        (1-Q_i)
\;\,    \L{-\mu}.$$
\end{lem}
\begin{proof}
    From Lemma \ref{lem:minuscule_correspondence}, we have 
\begin{align*}
\Psi\left(\prod_{\substack{i \in I \\ \langle \mu, \alpha_{i}^{\vee} \rangle = -1}} (1-\t{\alpha_i^\lor})
\;\,\COQl{w_\circ \mu}\right) 
=\CO_{G/B}({-}\mu).
\end{align*}

On the other hand, we can compute the same element as follows: 
\begin{align*}
&\Psi\left(
\prod_{\substack{i \in I \\ \langle \mu, \alpha_{i}^{\vee} \rangle = -1}} (1-\t{\alpha_i^\lor})
\;\,\COQl{w_\circ \mu}\right) \\
&=   
\prod_{\substack{i \in I \\ \langle \mu, \alpha_{i}^{\vee} \rangle = -1}} (1-Q_i)
\;\,\Psi\left(\COQl{w_\circ \mu}\right)
\quad \text{by \eqref{eq:Psi_COQ}}\\
&=\prod_{\substack{i \in I \\ \langle \mu, \alpha_{i}^{\vee} \rangle = -1}} (1-Q_i)
\;\,\L{-\mu},
\end{align*}
where in the last equality we used 
Lemma \ref{lem:Phi_sends_O_to_L}.
\end{proof}

\begin{proof}
[Proof of Theorem \ref{thm:main}]
Using the classical Borel presentation (see \cite{Mc})
\begin{equation}
    K_{T}(G/B) \simeq R(T) \otimes_{R(T)^{W}} R(T),
\end{equation}
we have the following identity in $QK_T(G/B)$
$$\sum_{\mu\in W\cdot\varpi}\CO_{G/B}(-\mu) = \chi_\varpi.$$ 
    Now the proof follows using Lemma \ref{cor_Omu_Lmu}.
\end{proof}

\section{Explicit minuscule-weight relations in classical types}\label{sec:example}
 In this section we write out the minuscule-weight relations in classical types. 

 For simplicity, we write
\[
  L_i:=L(-\varpi_i)=\CO_{G/B}(-\varpi_i)
  \qquad (i\in I).
\]
Here and below, negative powers are taken with respect to the quantum product.

\subsection{Type \(\mathrm A_{n-1}\)}

Let \(G=SL_n (\C) \). All the fundamental weights
\(\varpi_k\), \(1\leq k\leq n-1\), are minuscule. Set
\[
[n]:=\{1,\ldots,n\},\qquad
\varpi_0=\varpi_n=0,\qquad
\varepsilon_j:=\varpi_j-\varpi_{j-1}
\quad (1\leq j\leq n).
\]
For \(J\subset[n]\), put
\[
\varepsilon_J:=\sum_{j\in J}\varepsilon_j.
\]
Then
\[
W\cdot\varpi_k
=
\{\varepsilon_J\mid J\subset[n],\ |J|=k\}.
\]
For \(1\leq i\leq n-1\), we have
\[
\langle\varepsilon_J,\alpha_i^\vee\rangle
=
\begin{cases}
  1  & \text{if \(i\in J\) and \(i+1\notin J\)},\\
 -1  & \text{if \(i\notin J\) and \(i+1\in J\)},\\
  0  & \text{otherwise}.
\end{cases}
\]
Set \(L_0=L_n=1\). From the definition of \(L(\lambda)\), we obtain
\[
L(-\varepsilon_J)
=
\prod_{j\in J}L_jL_{j-1}^{-1}.
\]
Therefore, Theorem~\ref{thm:main} gives
\begin{equation}
\sum_{\substack{J\subset[n]\\ |J|=k}}
\left(
\prod_{\substack{1\leq i\leq n-1\\
i\notin J,\ i+1\in J}}
(1-Q_i)
\right)
\prod_{j\in J}L_jL_{j-1}^{-1}
=
\chi_{\varpi_k}.
\label{eq:An-minuscule}
\end{equation}

We compare this relation with the type \(\mathrm A\) presentation
in \cite{MNS1}. Set \(Q_0=Q_n=0\), and write
\(z_j:=1-x_j\) in the notation of \cite{MNS1}. The variable \(z_j\)
represents the modified line-bundle class
\[
z_j
\longmapsto
(1-Q_j)^{-1}\CO_{G/B}(-\varepsilon_j).
\]
Applying Lemma~\ref{cor_Omu_Lmu} to the weight \(\varepsilon_j\), we obtain
\[
\CO_{G/B}(-\varepsilon_j)
=
(1-Q_{j-1})L(-\varepsilon_j)
=
(1-Q_{j-1})L_jL_{j-1}^{-1}.
\]

Thus the variables \(z_j\) and \(L_j\) are related by
\begin{equation}
z_j
=
\frac{1-Q_{j-1}}{1-Q_j}
L_jL_{j-1}^{-1}
\qquad (1\leq j\leq n).
\label{eq:An-z-L}
\end{equation}
In terms of the variables \(z_j\), the relation in \cite{MNS1} is
\[
\sum_{\substack{J\subset[n]\\ |J|=k}}
\left(
\prod_{\substack{1\leq i\leq n-1\\
i\in J,\ i+1\notin J}}
(1-Q_i)
\right)
\prod_{j\in J}z_j
=
\chi_{\varpi_k}.
\]
Using \eqref{eq:An-z-L}, we have
\[
\left(
\prod_{\substack{1\leq i\leq n-1\\
i\in J,\ i+1\notin J}}
(1-Q_i)
\right)
\prod_{j\in J}z_j
=
\left(
\prod_{\substack{1\leq i\leq n-1\\
i\notin J,\ i+1\in J}}
(1-Q_i)
\right)
\prod_{j\in J}L_jL_{j-1}^{-1}.
\]
Hence the relation in \cite{MNS1} agrees exactly with
\eqref{eq:An-minuscule}.

For example, when \(G=SL_3 (\C)\), the relations corresponding to
\(\varpi_1\) and \(\varpi_2\) are, respectively,
\[
L_1+(1-Q_1)L_1^{-1}L_2+(1-Q_2)L_2^{-1}
=
\chi_{\varpi_1}
\]
and
\[
L_2+(1-Q_2)L_1L_2^{-1}+(1-Q_1)L_1^{-1}
=
\chi_{\varpi_2}.
\]

\subsection{Type \(\mathrm C_n\)}
The relation for 
$(\mathrm{C}_n,\vpi{1})$ reads
\begin{equation}
\sum_{i=1}^n
(1-Q_{i-1})L_iL_{i-1}^{-1}
+
\sum_{i=1}^n
(1-Q_i)L_{i-1}L_i^{-1}
=
\chi_{\varpi_1},
\label{eq:B_n-n}
\end{equation}
with $Q_0=0, L_0=1.$ In 
Kouno and Naito \cite{KN},
it was proved that the following relation holds in $QK_T({Sp}_{2n} (\C) /B)$:
\begin{equation}
F_1:=\sum_{i=1}^n(1-Q_i)z_i
+\sum_{i=1}^n(1-Q_{i-1})z_i^{-1}=\chi_{\vpi{1}}.
\label{eq:F_1_KN}
\end{equation}
The variables $z_i$ in \cite[\S~6.2]{KN} represent 
certain classes of modified line bundles as follows:
\begin{equation}
z_i\mapsto (1-Q_i)^{-1}\CO(-\varepsilon_i)\quad\text{and}
\quad
z_i^{-1}\mapsto (1-Q_{i-1})^{-1}\CO(\varepsilon_i)\quad (1\le i\le n),
\end{equation}
where $\varepsilon_{i} := \vpi{i} - \vpi{i-1}$ for $1 \le i \le n$ with $\vpi{0} = 0$. 
Thus we can identify 
the variables $z_i$ and $L_i$ by the equations
\begin{equation}
 z_i=({1-Q_i})^{-1}{(1-Q_{i-1})}
{L_i}L_{i-1}^{-1}.
\end{equation}
Then 
the relations \eqref{eq:B_n-n}
and \eqref{eq:F_1_KN}
coincide.

\subsection{Type $\mathrm{B}_n$}
In type $\mathrm{B}_n$, 
the weight $\vpi{n}$ corresponds to the spin representation.
We record the relations explicit for low
ranks. 
\begin{example}
$(\mathrm{B}_2,\vpi{2})$
\[L_2
+(1-Q_2)L_1L_2^{-1}
+(1-Q_1)L_1^{-1}L_2
+(1-Q_2)L_2^{-1}
=
\chi_{\varpi_2}.\]
\end{example}
\begin{example}
$(\mathrm{B}_3,\vpi{3})$
\[
\begin{aligned}
&
L_3
+(1-Q_3)L_2L_3^{-1}
+(1-Q_2)L_1L_3L_2^{-1}
+(1-Q_3)L_1L_3^{-1} \\
&\quad
+(1-Q_1)L_3L_1^{-1}
+(1-Q_1)(1-Q_3)L_2L_1^{-1}L_3^{-1} \\
&\quad
+(1-Q_2)L_3L_2^{-1}
+(1-Q_3)L_3^{-1}
=
\chi_{\varpi_3}.
\end{aligned}
\]    
\end{example}

\subsection{Type $\mathrm{D}_n$}

In type $\mathrm{D}_n$, there are three
minuscule fundamental weights
$$\vpi{1},\vpi{n-1},\vpi{n}.$$ The weights \(\varpi_{n-1}\) and
\(\varpi_n\) correspond to the two half-spin representations, while
\(\varpi_1\) corresponds to the vector representation.
\begin{example}
Type $\mathrm{D}_4$.
We write explicitly the relations
for \(\varpi_1\) and \(\varpi_4\): 
\begin{equation}
\label{eq:D4-varpi1-minuscule-relation}
\begin{aligned}
&
L_1
+
(1-Q_1)L_2L_1^{-1}
+
(1-Q_2)L_4L_2^{-1}L_3
+
(1-Q_3)L_4L_3^{-1}
+
(1-Q_4)L_4^{-1}L_3
\\
&\quad
+
(1-Q_4)(1-Q_3)L_4^{-1}L_2L_3^{-1}
+
(1-Q_2)L_2^{-1}L_1
+
(1-Q_1)L_1^{-1}
=
\chi_{\vpi{1}},\\
 &
 L_4
 +
 (1-Q_4)L_2L_4^{-1}
 +
 (1-Q_2)L_1L_2^{-1}L_3
 +
 (1-Q_3)L_1L_3^{-1}
 +
 (1-Q_1)L_1^{-1}L_3
 \\
 &\quad
 +
 (1-Q_1)(1-Q_3)L_1^{-1}L_2L_3^{-1}
 +
 (1-Q_2)L_2^{-1}L_4
 +
 (1-Q_4)L_4^{-1}
 =
 \chi_{\vpi{4}}.
\end{aligned}
\end{equation}
They are carried to each other by
the Dynkin diagram automorphism interchanging the outer nodes \(1\)
and \(4\).  The remaining minuscule relation, corresponding to
\(\varpi_3\), is obtained in the same way by triality, and will be
omitted.
\end{example}

\end{document}